\documentclass[dvipdfmx]{amsart}
\usepackage{amsmath}
\usepackage{amssymb} 
\usepackage{amscd} 
\usepackage{graphicx} 
\usepackage{epsfig}
\usepackage[dvips]{color}
\usepackage{comment}

\newtheorem{theorem}{Theorem}[section]
\newtheorem{lemma}[theorem]{Lemma}

\newtheorem{claim}[theorem]{Claim}

\newtheorem{conjecture}[theorem]{Conjecture}

\newtheorem{question}[theorem]{Question}

\theoremstyle{definition}

\newtheorem{example}[theorem]{Example}

\newtheorem{remark}[theorem]{Remark}

\numberwithin{equation}{section}
\numberwithin{figure}{section}
\numberwithin{table}{section}

\renewcommand{\(}{\textup{(}}
\renewcommand{\)}{\textup{)}}

\begin{document}
\baselineskip 14pt

\title{Generalized torsion  and decomposition of $3$--manifolds}
\author[T.Ito]{Tetsuya Ito}
\address{Department of Mathematics, Kyoto University, Kyoto 606-8502, JAPAN}
\email{tetitoh@math.kyoto-u.ac.jp}
\thanks{The first named author has been partially supported by JSPS KAKENHI Grant Number JP15K17540 and JP16H02145.}

\author[K. Motegi]{Kimihiko Motegi}
\address{Department of Mathematics, Nihon University, 
3-25-40 Sakurajosui, Setagaya-ku, 
Tokyo 156--8550, Japan}
\email{motegi@math.chs.nihon-u.ac.jp}
\thanks{The second named author has been partially supported by JSPS KAKENHI Grant Number JP26400099 and Joint Research Grant of Institute of Natural Sciences at Nihon University for 2018. }

\author[M. Teragaito]{Masakazu Teragaito}
\address{Department of Mathematics and Mathematics Education, Hiroshima University, 
1-1-1 Kagamiyama, Higashi-Hiroshima, 739--8524, Japan}
\email{teragai@hiroshima-u.ac.jp}
\thanks{The third named author has been partially supported by JSPS KAKENHI Grant Number JP16K05149.}

\dedicatory{}

\begin{abstract}
A nontrivial element in a group  is a \textit{generalized torsion element} 
if some nonempty finite product of its conjugates is the identity. 
We prove that any generalized torsion element in a free product of torsion-free groups is conjugate to a generalized torsion element in some factor group. 
This implies that the fundamental group of a compact orientable $3$--manifold $M$ has a generalized torsion element 
if and only if the fundamental group of some prime factor of $M$ has a generalized torsion element. 
On the other hand, 
we demonstrate that there are infinitely many toroidal $3$--manifolds whose fundamental group has a generalized torsion element, 
while the fundamental group of each decomposing piece has no such elements. 
\end{abstract}

\maketitle

{
\renewcommand{\thefootnote}{}
\footnotetext{
}

\setcounter{tocdepth}{1}

\section{Introduction}
\label{introduction}

A nontrivial element in a group $G$ is a \textit{generalized torsion element} 
if some nonempty finite product of its conjugates is the identity.
As a natural generalization of the order of torsion element, the \emph{order} of a generalized torsion element $g$ is defined as 
\[
\min\{ k \geq 1 \: | \: \exists x_{1},\ldots,x_k \in G \mbox{ s.t. } (x_1 g x_1^{-1})(x_2 g x_2^{-1}) \cdots(x_{k}g x_{k}^{-1})=1\}. \]
 In the literature \cite{BL,LMR,MR1,MR2}, a group without generalized torsion element is called an $R^*$--group or a $\Gamma$--torsion-free group.

A group $G$ is \textit{bi-orderable} if $G$ admits a total ordering $<$ which is invariant under the multiplication from both left and right sides.
That is, if $g<h$, then $agb<ahb$ for any $g,h,a,b\in G$. Such a ordering is called a \emph{bi-ordering} on $G$.
In this paper, the trivial group $\{1\}$ is considered to be bi-orderable.

A generalized torsion element plays as a fundamental obstruction for a group to be bi-orderable, 
since a bi-orderable group has no generalized torsion element. 
Although the converse does not hold in general \cite[Chapter 4]{MR2}, 
the converse is known to be true for some classes of groups.
In \cite{MT2} the second and the third authors demonstrate that the converse is true for the fundamental group of any non-hyperbolic geometric $3$--manifold, 
and propose the following conjecture. 

\begin{conjecture}[\cite{MT2}]
\label{conj:bo}
Let $G$ be the fundamental group of a $3$--manifold. 
Then,
$G$ is bi-orderable if and only if
$G$ has no generalized torsion element.
\end{conjecture}

Recall that every compact orientable $3$--manifold $M$ decomposes uniquely as a connected sum 
$M = M_1 \sharp M_2 \sharp \cdots \sharp M_n$ of $3$--manifolds $M_i$ which are prime in the sense that they can be 
decomposed as connected sums only in the trivial way $M_i = M_i \sharp S^3$ \cite{Kne, Milnor}. 
Note that an orientable, prime $3$--manifold is irreducible except for $S^2 \times S^1$. 
Beyond the prime decomposition, 
we have a further canonical decomposition of irreducible compact orientable $3$--manifolds, 
decomposing along tori rather than $2$--spheres. 
Precisely, for a compact irreducible orientable $3$--manifold $M$ 
there exists a collection $\mathcal{T} = \{ T_1, \dots, T_k \}$ of disjoint incompressible tori in $M$ such that each component of $M - \mathrm{int}N(\cup T_i)$ is either atoroidal or a Seifert fiber space, 
and a minimal such collection $\mathcal{T}$ is unique up to isotopy \cite{JS, Jo}; see also \cite{Hat}. 
Thurston's uniformization theorem \cite{Thurston,MB}, 
together with the Torus Theorem \cite{JS, Jo}, 
shows that an atoroidal piece which is not a Seifert fiber space admits a complete hyperbolic structure of finite volume. 
  
In group theoretic language, 
the prime decomposition corresponds to a decomposition by a free product, 
and the torus decomposition corresponds to a graph of groups with each edge group is isomorphic to $\mathbb{Z} \oplus \mathbb{Z}$, 
which generalizes a free product with amalgamation. 
Let $M$ be a compact orientable $3$--manifold. 
A generalized torsion element in the fundamental group of a prime factor or a decomposing piece (with respect to the torus decomposition) of $M$ becomes a generalized torsion element of $\pi_1(M)$. 

It is plausible to address the converse. 

\begin{question}
If $\pi_1(M)$ has a generalized torsion element $g$, 
then does it arise from the fundamental group of a prime factor or a decomposing piece? 
More precisely, 
is $g$ conjugate to a generalized torsion element in the fundamental group of a prime factor or a decomposing piece? 
\end{question}

The aim of this article is to answer this question in the positive for prime decompositions, 
while in the negative for torus decompositions. 

\begin{theorem}
\label{prime_decomposition}
Let $M$ be a compact, orientable $3$--manifold with the prime decomposition 
$M = M_1 \sharp M_2 \sharp \cdots \sharp M_n$. 
\begin{enumerate}
\item 
$\pi_1(M)$ has a generalized torsion element if and only if 
$\pi_1(M_i)$ has a generalized torsion element for some $i = 1, \dots, n$.

\item
If $\pi_1(M_i)$ is torsion-free for all $1 \le i \le n$, 
then any generalized torsion element in $\pi_1(M)$ is conjugate to a generalized torsion element in 
$\pi_1(M_i)$ for some $i$. 
\end{enumerate}
\end{theorem}

It should be remarked that in the second assertion of Theorem~\ref{prime_decomposition} we need the extra assumption on absence of torsion elements in $\pi_1(M_i)$. 

\begin{remark}
Let $M_1=L(p,1)$ and $M_2=L(q,1)$ be lens spaces $(p,q>1)$.
Then 
$\pi_1(M_1) = \langle a\ |\ a^p = 1 \rangle$, 
$\pi_1(M_2) = \langle b\ |\ b^q = 1 \rangle$ and 
$\pi_1(M_1 \sharp M_2) = \langle a, b\ |\ a^p = b^q = 1 \rangle \cong \mathbb{Z}_p * \mathbb{Z}_q$. 
For an element $ab$, we have
\[ (ab) \big(b^{q-1}(ab)b^{1-q}\big)\big(b^{q-2}(ab)b^{2-q}\big) \cdots \big(b(ab)b^{-1}\big)=a^{q}. \]
In the abelianization $\pi_1(M_1 \sharp M_2) \slash [\pi_1(M_1 \sharp M_2),\pi_1(M_1 \sharp M_2)]\cong \mathbb{Z}_p \oplus \mathbb{Z}_q$, $ab$ is of order $\mathrm{lcm}(p,q)$, the least common multiple of $p$ and $q$. 
Therefore $ab$ is a generalized torsion element of order of $\mathrm{lcm}(p,q)$.
However, $ab$ is not conjugate into $\pi_1(M_i)$ for $i = 1, 2$. 
\end{remark}

Theorem~\ref{prime_decomposition} immediately follows from a more general, purely algebraic result below, 
which solves the open problem due to V.M.~Kopytov and N.Ya.~Medvedev \cite[Problem 16.49]{KM} affirmatively. 

\begin{theorem}
\label{free_product}
Let $G = G_1 * \cdots * G_n$ be the free product of torsion-free groups. 
Suppose that $g \in G$ is a generalized torsion element. 
Then there is a generalized torsion element $g_i \in G_i$ which is conjugate to $g$ for some $i = 1, \dots, n$. 
Hence, 
$G$ has a generalized torsion element if and only if 
$G_i$ has a generalized torsion element for some $i$. 
\end{theorem}

Let us turn to torus decompositions. 
The situation is quite different from prime decompositions. 
Actually we have: 

\begin{theorem}
\label{torus_decomposition}
For arbitrary $n > 1$, there exists an infinite family of $3$--manifolds $M$ such that:
\begin{itemize}
\item $M$ has a torus decomposition $\displaystyle M = M_1 \cup M_2 \cup \cdots \cup M_n$ 
with $n$ decomposing pieces. 
\item None of  $\pi_1(M_i)$ \($1 \le i \le n$\) has a generalized torsion element. 
\item $\pi_1(M)$ has a generalized torsion element.
\end{itemize}
\end{theorem}

Theorem~\ref{torus_decomposition} says that there is a ``global'' generalized torsion element with respect to torus decompositions.  

\medskip

The paper is organized as follows. 
In Section~\ref{section:free_product} we focus on prime decompositions and prove Theorem~\ref{free_product}. 
A key ingredient is a stable commutator length (scl) and we give an upper bound for the stable commutator length of generalized torsion elements 
(Lemma~\ref{scl_generalized_torsion}) in Subsection~\ref{scl_genralized_torsion}.   
Section~\ref{JSJ_decomposition} is devoted to a proof of Theorem~\ref{torus_decomposition}. 

\bigskip

\section{Generalized torsion and free product}
\label{section:free_product}

In this section we prove Theorem \ref{free_product}. 
Our main tool is the stable commutator length (scl, in short). For basics of scl we refer to \cite{Ca}.
Throughout the section, $G$ denotes a group.
\medskip 

\subsection{Basic facts on stable commutator length}

For $g \in [G,G]$ the \emph{commutator length} $\mathrm{cl}(g)$ is the smallest number of commutators in $G$ whose product is equal to $g$. 
The stable commutator length $\mathrm{scl}(g)$ of $g \in [G,G]$ is defined to be the limit
\begin{equation}
\label{def1}
\mathrm{scl}(g) = \lim_{n \to \infty} \frac{\mathrm{cl}(g^n)}{n}.
\end{equation}

Since $\mathrm{cl}(g^n)$ is non-negative and subadditive, Fekete's subadditivity lemma shows that 
this limit exists.  

We will extend (\ref{def1}) to the stable commutator length $\mathrm{scl}(g)$ for an element $g$ which is not necessarily in $[G, G]$ as 
\begin{equation}
\label{def2}
\mathrm{scl}(g) = \begin{cases}
\frac{\textrm{scl}(g^{k})}{k} & \mbox{ if } g^{k} \in [G,G] \mbox{ for some } k > 0,\\
\infty & \mbox{otherwise}.
\end{cases}
\end{equation}

By definition (\ref{def1}), it is easily observe that if $g^k,\, g^{\ell} \in [G, G]$, 
then $\displaystyle \frac{\mathrm{scl}(g^k)}{k} = \frac{\mathrm{scl}(g^{\ell})}{\ell}$ for any $k,\, \ell > 0$. 
So in (\ref{def2}) $\mathrm{scl}(g)$ is independent of the choice of $k > 0$ such that $g^k \in [G, G]$. 
In particular, for later use we note the following. 

\begin{lemma}
\label{scl_g^k}
For any $g \in G$ and $k > 0$,  
we have $\mathrm{scl}(g^{k})=k\,\mathrm{scl}(g)$. 
\end{lemma}

\begin{proof}
If $g^{\ell} \not\in [G, G]$ for all $\ell > 0$, 
then $\mathrm{scl}(g) = \mathrm{scl}(g^{\ell}) = \infty$. 
We assume $g^{\ell} \in [G, G]$ for some $\ell > 0$. 
Then $g^{\ell k} \in [G, G]$, 
and 
$\displaystyle \mathrm{scl}(g) = \frac{\mathrm{scl}(g^{\ell})}{\ell} 
= \frac{\mathrm{scl}(g^{\ell k})}{\ell k} 
= \frac{1}{k}\, \frac{\mathrm{scl}((g^{k})^{\ell})}{\ell} 
=  \frac{\mathrm{scl}(g^{k})}{k}$. 
Thus we have $\mathrm{scl}(g^k) = k\, \mathrm{scl}(g)$. 
\end{proof}

When we need to specify the underlying group $G$, 
we write $\mathrm{scl}_G(g)$ instead of $\mathrm{scl}(g)$. 

A map $\phi:G \rightarrow \mathbb{R}$ is a \textit{quasimorphism} if 
\[ D(\phi) = \sup_{g,h \in G} |\phi(gh)-\phi(g)-\phi(h)|\]
is finite. 
The value $D(\phi)$ is called the \emph{defect} of $\phi$. 
Furthermore if $\phi(g^{n})=n\phi(g)$ holds for all $g \in G$ and $n \in \mathbb{Z}$, 
then we say that $\phi$ is \textit{homogeneous}. 
In particular, a homogeneous quasimorphism is antisymmetric, i.e. it satisfies $\phi(g^{-1})=-\phi(g)$. 
Note that a homogeneous quasimorphism $\phi$ is a class function, 
i.e. $\phi(g)=\phi(h^{-1}gh)$ for every $g,h \in G$; see \cite[2.2.3]{Ca}.  

\begin{theorem}[Bavard's duality theorem \cite{Ba}]
\label{Bavard duality}
For $g \in [G,G]$ we have
\[ \mathrm{scl}(g) = \sup_{\phi} \frac{|\phi(g)|}{2D(\phi)},\]
where $\phi$ runs all the homogenous quasimorphisms on $G$ which are not homomorphisms. 
\end{theorem}

\begin{example}[p.19 in \cite{Ca}]
\label{example:single_commutator}
For any homogeneous quasimorphism $\phi$, 
\begin{align*}
|\phi([a, b])| 
& = |\phi(a^{-1}b^{-1}ab) |\\
& = |\phi(a^{-1}b^{-1}ab) - \phi(a^{-1}b^{-1}a) - \phi(b) +  \phi(a^{-1}b^{-1}a) + \phi(b)|\\
& =  |\phi(a^{-1}b^{-1}ab) - \phi(a^{-1}b^{-1}a) - \phi(b) +  \phi(a^{-1}b^{-1}a) - \phi(b^{-1})|\\
& \leq  |\phi(a^{-1}b^{-1}ab) - \phi(a^{-1}b^{-1}a) - \phi(b)| +  |\phi(a^{-1}b^{-1}a) - \phi(b^{-1})|\\
& = |\phi(a^{-1}b^{-1}ab) - \phi(a^{-1}b^{-1}a) - \phi(b)|\\
& \leq D(\phi).
\end{align*}
Hence $\textrm{scl}([a, b]) \leq \frac{1}{2}$.
\end{example}

\medskip 

\subsection{Stable commutator length and generalized torsion elements}
\label{scl_genralized_torsion}

Obviously a torsion element $g$ has $\mathrm{scl}(g) = 0$. 
For a generalized torsion element we have the following. 

\begin{theorem} 
\label{scl_generalized_torsion}
If $g \in G$ is a generalized torsion element of order $k$, then $\mathrm{scl}(g) \leq \frac{k-2}{2k} < \frac{1}{2}$.
\end{theorem}

\begin{proof}[Proof of Theorem~\ref{scl_generalized_torsion}]
Since $g$ is a generalized torsion of order $k$, 
there exist $x_{1},\ldots,x_{k} \in G$ such that
\[
(x_1 g x_1^{-1})(x_2 g x_2^{-1})(x_3 g x_3^{-1}) \cdots(x_{k}g x_{k}^{-1})=1,
\]
and hence $g^k = 1$ in $G/[G, G]$ and $g^{k} \in [G,G]$. 
By taking conjugate by $x_{1}^{-1}$ and putting $y_i = x_1^{-1}x_i$, we have
\[ g^{-1} = (y_2 g y_2^{-1})(y_3 g y_3^{-1}) \cdots(y_{k}g y_{k}^{-1}).\]
We estimate $\mathrm{scl}(g^{k})$ by using Bavard's duality.

Let $\phi$ be a homogeneous quasimorphism which is not a homomorphism. We put 
\begin{align*}
A_2 &= \phi((y_2 g y_2^{-1}) \cdots(y_{k}g y_{k}^{-1})) - \phi(y_2 g y_2^{-1}) - \phi((y_3 g y_3^{-1}) \cdots(y_{k}g y_{k}^{-1})), \\
A_3 &= \phi((y_3 g y_3^{-1}) \cdots(y_{k}g y_{k}^{-1})) - \phi(y_3 g y_3^{-1}) - \phi((y_4 g y_4^{-1}) \cdots(y_{k}g y_{k}^{-1})), \\
& \, \; \;  \vdots \\
A_{k-1} &=  \phi((y_{k-1} g y_{k-1}^{-1})(y_{k} g y_{k}^{-1})) - \phi(y_{k-1} g y_{k-1}^{-1}) - \phi(y_k g y_k^{-1}).
\end{align*}
Then $|A_i| \le D(\phi)$, and 
\[ A_2 + A_3 + \cdots + A_{k-1} = \phi((y_2 g y_2^{-1}) \cdots(y_{k}g y_{k}^{-1})) - \sum_{i=2}^{k} \phi(y_i g y_i^{-1}).\]
Since $\phi(y_i g y_i^{-1})=\phi(g)$, we conclude that 
\[ |\phi((y_2 g y_2^{-1})(y_3 g y_3^{-1}) \cdots(y_{k}g y_{k}^{-1})) - (k-1)\phi(g)| \leq (k-2)D(\phi). \]
On the other hand,  
\[\phi((y_2 g y_2^{-1})(y_3 g y_3^{-1}) \cdots(y_{k}g y_{k}^{-1})) = \phi(g^{-1}) = -\phi(g),\] 
so we get
\[  |\phi(g^{k})| = k|\phi(g)| \leq (k-2)D(\phi) \iff \frac{|\phi(g^{k})|}{2D(\phi)} \leq \frac{k-2}{2} \]
for every homogeneous quasimorphism $\phi$. 
By Bavard's duality this shows that $\mathrm{scl}(g^{k}) \leq \frac{k-2}{2}$. 
Since $\mathrm{scl}(g^{k}) = k\, \mathrm{scl}(g)$ (Lemma~\ref{scl_g^k}), 
we obtain the desired inequality $\textrm{scl}(g) \leq \frac{k-2}{2k} < \frac{1}{2}$. 
\end{proof}

\medskip

\subsection{Proof of Theorem~\ref{free_product}}
Let us recall the following result of Chen \cite{Chen}; 
originally Theorem A in \cite{Chen} requires a stronger condition ``$g \in [G, G]$'', 
but as the author remarked we can weaken this condition to  ``$g^n \in [G, G]$ for some $n \ge 1$''. 
Since $\mathrm{scl}_G(g) = \infty$ when $g^n \not\in [G, G]$ for any $n > 0$, 
we state the result in the following form. 

\begin{theorem}\cite[Theorem A]{Chen}
\label{theorem:Chen}
Let $G = G_1 * \cdots * G_n$ be the free product of torsion-free groups $G_i$. 
If $g \in G$ is not conjugate into $G_i$ for any $i$ $(1 \le i \le n)$, 
then $\mathrm{scl}_{G}(g) \geq \frac{1}{2}$.
\end{theorem}

\begin{proof}[Proof of Theorem~\ref{free_product}]
Let $g$ be a generalized torsion element in $G$ (of order $k$). 
Since $g$ is a generalized torsion of order $k$, 
there exist $x_{1},\ldots,x_{k} \in G$ such that

\begin{equation}
\label{g-torsion}
(x_1 g x_1^{-1})(x_2 g x_2^{-1})(x_3 g x_3^{-1}) \cdots(x_{k}g x_{k}^{-1})=1,
\end{equation}
and hence $g^k = 1$ in $G/[G, G]$ and $g^{k} \in [G,G]$.  
Assume for a contradiction that $g$ is not conjugate into $G_i$ for any $i$ $(1 \le i \le n)$.  
Then Theorem \ref{theorem:Chen} shows that $\mathrm{scl}(g) \ge \frac{1}{2}$. 
On the other hand, Theorem~\ref{scl_generalized_torsion} asserts that $g < \frac{1}{2}$. 
This is a contradiction. 
Hence, $g$ is conjugate into $G_i$ for some $i$ ($1 \le i \le n$).  
Without loss of generality we many assume $y g y^{-1} = g_1 \in G_1$ for some $y \in G$. 
In the remaining we show that $g_1$ is a generalized torsion element of order $k$ in $\pi_1(M_1)$. 
Following (\ref{g-torsion}), we have 
\begin{align*}
1 
& = y(x_1 g x_1^{-1})(x_2 g x_2^{-1})(x_3 g x_3^{-1}) \cdots(x_{k}g x_{k}^{-1})y^{-1} \\ 
& = (yx_1y^{-1}) (y g y^{-1}) (y x_1^{-1} y^{-1}) \cdots  (yx_k y^{-1}) (y g y^{-1}) (y x_k^{-1} y^{-1}) \\
& =  (yx_1y^{-1}) g_1 (y x_1 y^{-1})^{-1} \cdots  (yx_k y^{-1}) g_1 (y x_k y^{-1})^{-1}.
\end{align*}

Using the natural projection $p : G_1 * \cdots * G_n \to G_1$, 
this gives
\begin{align*}
G_1  \ni 1 
& = p( (yx_1y^{-1}) g_1 (y x_1 y^{-1})^{-1} \cdots  (yx_k y^{-1}) g_1 (y x_k y^{-1})^{-1} ) \\
& = (y_1x_{1, 1}y_1^{-1}) g_1 (y_1 x_{1, 1} y_1^{-1})^{-1} \cdots  (y_1x_{1, k} y^{-1}) g_1 (y x_{1, k} y^{-1})^{-1},
\end{align*}
where $y_1 = p(y) \in G_1$ and $x_{1, j} = p (x_j) \in G_1$ ($j = 1, \dots,k$).  

Hence we have  
\[
(y_1x_{1, 1}y_1^{-1}) g_1 (y_1 x_{1, 1} y_1^{-1})^{-1} \cdots  (y_1x_{1, k} y^{-1}) g_1 (y x_{1, k} y^{-1})^{-1} = 1. 
\]
This shows that $g_1$ is a generalized torsion element of order at most $k$ in $G_1$.  
If $g_1$ has order less than $k$, 
then $g = y^{-1}g_1 y$ has also order less than $k$, 
a contradiction. 
So $g_1$ is a generalized torsion element of order $k$. 

Since a torsion element is obviously a generalized torsion element, 
and a generalized torsion element in $G_i$ is also a generalized torsion element in $G = G_1* \cdots *G_n$, 
the last assertion immediately follows. 
\end{proof}

\bigskip

\section{Torus decompositions and generalized torsion elements}
\label{JSJ_decomposition}

In Section~\ref{prime_decomposition} we saw that a generalized torsion nicely behaves under prime decomposition. 
In this section we prove Theorem \ref{torus_decomposition}, 
which illustrates that the behavior of generalized torsion under torus decomposition is completely different. 

\medskip

\subsection{Creation of generalized torsion elements via Dehn fillings}
\label{Dehn_filling}

We introduce a useful way to create a generalized torsion element in a $3$--manifold group. 
To explain our construction, 
we prepare some notions.
A \emph{singular spanning disk} of a knot $K$ in $S^{3}$ is a smooth map $\Phi:D^{2} \rightarrow S^{3}$ (or, its image) such that $\Phi(\partial D) = K$ and that $K$ intersects $\Phi(\mathrm{int}D)$ transversely in finitely many points. 
Each intersection point $\Phi(\mathrm{int}D \cap K)$ has a sign according to the orientations. We say that $\Phi(D)$ a $(p, q)$--\textit{singular spanning disk} if 
$K$ intersects $\Phi(\mathrm{int}D)$ positively in $p$ points and negatively in $q$ points.

Such a disk appeared in early 3-dimensional topology -- Dehn's lemma \cite{Papa, Hom} says that $K$ has a $(0, 0)$--singular spanning disk if and only if $K$ is the trivial knot. 
More generally, $K$ has a $(p, q)$--singular spanning disk with $0\leq p,q \leq 1$ if and only if $K$ is still the trivial knot \cite{Ito1}. 

Let us recall the following theorem in \cite{IMT_filling}, 
which shows that a singular spanning disk can be used to create a generalized torsion.

\begin{theorem}[\cite{IMT_filling}]
\label{thm:g_torsion_span_disk}
Let $K$ be a knot in $S^3$. 
If $K$ has a $(p, 0)$--singular spanning disk, 
then the image of a meridian $\mu$ of $K$ in $\pi_1(K(\frac{m}{n}))$ is a generalized torsion element of order $m$ whenever $\frac{m}{n} \ge p$. 
Similarly, if $K$ has a $(0, q)$--singular spanning disk,  
then the image of a meridian $\mu$ in $\pi_1(K(\frac{m}{n}))$ is a generalized torsion element  of order $m$ whenever $\frac{m}{n}\le -q$. 
\end{theorem}

A typical and fundamental example of $(p, 0)$-- or $(0, q)$--singular spanning disk is a clasp disk having the same sign of clasps, 
which we call a \emph{coherent clasp disk}. 

\medskip

\subsection{Construction of toroidal $3$--manifolds with only global generalized torsion elements}
\label{construction}

\begin{proof}[Proof of Theorem \ref{torus_decomposition}]
We take a connected sum $K = K_{p_1} \sharp K_{p_2} \sharp \cdots \sharp K_{p_n}$ 
of positive twist knots $K_{p_i}$ with $p_i > 0$ for $i = 1, \dots, n$ (Figure~\ref{connected_sum_twist_knots_n}). 

\begin{figure}[htpb]
\includegraphics*[width=0.9\textwidth]{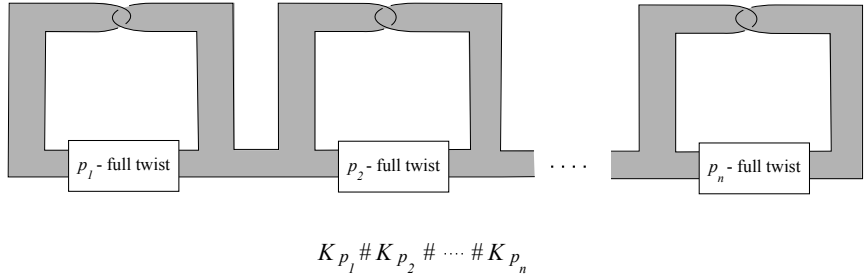}
\caption{$K_{p_1} \sharp K_{p_2} \sharp \cdots \sharp K_{p_n}$ bounds $(2n,0)$-- singular spanning disk with $n$ positive clasps}
\label{connected_sum_twist_knots_n}
\end{figure}

Then $E(K)$ has the torus decomposition 
$E(K) = E(K_{p_1}) \cup  E(K_{p_2}) \cup \cdots \cup E(K_{p_n}) \cup X$, 
where $X$ is the $n$--fold composing space, i.e. $[\textrm{disk with $n$--holes}] \times S^1$. 
Note that $\partial X = \partial E(K) \cup \partial E(K_1) \cup \cdots \cup \partial E(K_n)$ and 
$K(m) = E(K_{p_1}) \cup  E(K_{p_2}) \cup \cdots \cup E(K_{p_n}) \cup X(m)$, 
where $X(m)$ is the $3$--manifold obtained by gluing $S^1 \times D^2$ so that $\{ * \} \times D^2$ is identified a simple closed curve 
on $\partial E(K) \subset \partial X$ with slope $m$. 
Since the regular fiber of $\partial X$ is a meridian of $K$, 
the dual knot $K^*$ of $K$ becomes a regular fiber in $X(m)$ for any integer $m$. 
This means that $X(m)$ is the $(n-1)$--fold composing space. 
Thus $K(m) = E(K_{p_1}) \cup  E(K_{p_2}) \cup \cdots \cup E(K_{p_n}) \cup X(m)$ gives the torus decomposition of $K(m)$ 
with $(n+1)$--decomposing pieces; 
when $n = 2$, then $X(m)$ is homeomorphic to $S^1 \times S^1 \times [0, 1]$ and hence 
$K(m) = E(K_{p_1}) \cup  E(K_{p_2})$, which has two decomposing pieces.  
We denote $K(m)$ by $M_{p_1, \dots, p_n, m}$ and, for simplicity, we assume $m > 0$ in the following. 

\begin{claim}
\label{no_torsion}
For each decomposing piece $E(K_{p_i})$ and $X(m)$, their fundamental group have no generalized torsion. 
\end{claim}

\begin{proof}
Since $p_i > 0$, \cite{CDN} shows that $\pi_1(E(K_{p_i}))$ is bi-orderable, and thus it has no generalized torsion element.  
Recall that $X(m)$ is a composing space which is a locally trivial circle bundle over a punctured disk, 
referring \cite[Theorem~1.5]{BRW} we see that $X(m)$ is also bi-orderable, 
and hence its fundamental group does not have a generalized torsion element. 
\end{proof}

\begin{claim}
\label{torsion}
$\pi_1(M)$ has a generalized torsion element. 
\end{claim}

\begin{proof}
Observe that $K = K_{p_1} \sharp K_{p_2} \sharp \cdots \sharp K_{p_n}$ bounds a coherent $(2n, 0)$--singular disk (with $n$ positive clasp singularities); 
see Figure~\ref{connected_sum_twist_knots_n}.  
Then Theorem~\ref{thm:g_torsion_span_disk} shows that the image of a meridian of $K$ in $\pi_1(K(m))$ is a generalized torsion element 
for $m \ge 2n$. 
\end{proof}

\begin{claim}
\label{classification}
$M_{p_1, \dots, p_n, m}$ is homeomorphic to $M_{q_1, \dots, q_{n'}, m'}$ 
if and only if $(p_1, \dots, p_n) = (q_1, \dots, q_{n'})$ up to permutation and $m = m'$. 
\end{claim}

\begin{proof}
The ``if'' part is obvious, so we prove the ``only if'' part.  
Since $H_1(K(m)) \cong \mathbb{Z}_m$, the second condition is immediate. 
The classification of twist knots using Alexander polynomials shows that 
$E(K_p)$ is homeomorphic to $E(K_q)$ only when $p = q$. 
Hence, the first condition follows from the uniqueness of the torus decomposition. 
\end{proof}

The infinite family $\{ M_{p_1, \dots, p_n, m} \}$ gives a desired family. 
\end{proof}

\begin{remark}\
\label{global_local}
\begin{enumerate}
\item
In the above construction, when $n > 2$, 
the fundamental group $\pi_1(K(m))$ is described by a ``graph of groups'' with vertex groups 
$\pi_1(E(K_{p_i}))$, $\pi_1(X(m))$ and edge groups $\pi_1(T_i) \cong \mathbb{Z} \oplus \mathbb{Z}$ with $T _i = \partial E(K_{p_i})$ rather than a usual free product with amalgamation. 
See Figure~\ref{graph_groups}. 
Actually a free product with amalgamation corresponds to a graph of groups on a graph with two vertices and one edge.  
\begin{figure}[htpb]
\includegraphics*[width=0.35\textwidth]{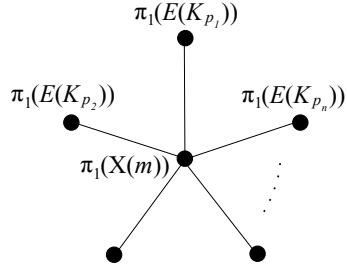}
\caption{$\pi_1(K(m))$ is expressed as a graph of groups.}
\label{graph_groups}
\end{figure}

\item
The generalized torsion element $g$ in $\pi_1(K(m))$ is the image of the meridian of $K = K_{p_1} \sharp K_{p_2} \sharp \cdots \sharp K_{p_n}$, 
so it is contained in $\pi_1(X(m))$. 
However, as shown in Claim~\ref{no_torsion}, 
$g$ cannot be a generalized torsion element in $\pi_1(X(m))$. 

\item
As we mentioned, the meridian of $E(K)$ is a regular fiber of $X$ and freely homotopic to a meridian of $E(K_{p_i})$ for each $i = 1, \dots, n$.  
Therefore after Dehn filling, the generalized torsion element $g$ is conjugate to $g_i$ in $\pi_1(E(K_{p_i}))$ for each $i$. 
However, $g_i$ cannot be a generalized torsion element in $\pi_1(E(K_{p_i}))$. 

\item
Roughly speaking, 
the generalized torsion element $g$ is ``global'' as a generalized torsion element, 
but itself is ``local'' in the sense that $g$ is conjugate into an edge group. 
\end{enumerate}
\end{remark}

Remark~\ref{global_local} leads us to ask: 

\begin{question}
Let $M$ be a compact orientable irreducible $3$--manifold 
and $\mathcal{T}$ a family of essential tori in $M$ which gives a torus decomposition of $M = M_1 \cup \cdots \cup M_n$. 
Suppose that $g$ is a generalized torsion element in $\pi_1(M)$. 
Then does either (1) or (2) hold?
\begin{enumerate}
\item
$g$ is conjugate to a generalized torsion element in $\pi_1(M_i)$ for some $i$. 

\item
If none of $\pi_1(M_i)$ $(1 \le i \le n)$ has a generalized torsion element, 
then $g$ is conjugate into an edge group $\pi_1(T)$ for some $T \in \mathcal{T}$. 
\end{enumerate}
\end{question}


\bigskip 

\begin{thebibliography}{99}
\bibitem{Ba}
C. Bavard, 
\textit{Longueur stable des commutateurs.}
Enseign.\ Math.\ (2) \textbf{37} (1991), 109--150.

\bibitem{BL}
V. V. Bludov and E. S. Lapshina,
\textit{On ordering groups with a nilpotent commutant} (in Russian),
Sibirsk.\ Mat.\ Zh.\  \textbf{44} (2003),  no. 3, 513--520;  translation in  Siberian Math.\ J. \textbf{44}  (2003),  no. 3, 405--410

\bibitem{BRW} 
S. Boyer, D. Rolfsen and B. Wiest, 
Orderable 3-manifold groups, 
Ann.\ Inst.\ Fourier \textbf{55} (2005), 243--288. 

\bibitem{Ca}
D. Calegari,
\textit{scl},
MSJ Memoirs, 20. Mathematical Society of Japan, Tokyo, 2009. xii+209 pp.

\bibitem{Chen}
L. Chen, 
{\em Spectral gap of scl in free products,}
Proc.\ Amer.\ Math.\ Soc.\ \textbf{146} (2018), no. 7, 3143--3151.


\bibitem{CDN}
A. Clay, C. Desmarais and P. Naylor, 
\textit{Testing bi-orderability of knot groups,}
Canad.\ Math.\ Bull.\ \textbf{59} (2016), no. 3, 472--482. 

\bibitem{Hat}
A.E. Hatcher; 
Notes on basic $3$--manifold topology (2000), 
freely available at {\verb!http://www.math.cornell.edu/~hatcher!}

\bibitem{Hom}
T. Homma, 
On Dehn Lemma for three sphere, 
Yokohama Math.\ J.\ \textbf{5} (1957), 223--244.  

\bibitem{Ito1}
T. Ito, 
Framing functions and a strengthened version of Dehn's lemma.  
J. Knot Theory Ramifications \textbf{25} (2016), no. 6, 1650031, 5 pp. 

\bibitem{IMT_filling}
T. Ito, K. Motegi, and M. Teragaito, 
Generalized torsion and Dehn filling, 
in preparation. 

\bibitem{JS} 
W. Jaco and P. Shalen; 
Seifert fibered spaces in $3$--manifolds, 
Mem.\ Amer.\ Math.\ Soc.\ \textbf{21}  (1979),\ no. 220, viii+192. 

\bibitem{Jo} 
K. Johannson; 
Homotopy equivalences of $3$--manifolds with boundaries, 
Lect.\ Notes in Math.\ vol.\ 761,\ Springer-Verlag,\ 1979.  

\bibitem{KM}
E. Khukhro and V. Mazurov, 
\textit{Unsolved problems in group theory. The Kourovka notebook No.19},
arXiv:1401.0300v14

\bibitem{Kne}
H. Kneser; 
Geschlossene Fl\"achen in dreidimensionalen Mannigfaltigkeiten, 
Jber.\ Deutsch.\ Math.\ Verein.\ \textbf{38} (1929), 248--260. 


\bibitem{LMR}
P. Longobardi, M. Maj and A. Rhemtulla, 
\textit{On solvable $R^*$--groups},
J.\ Group Theory \textbf{6}  (2003),  no. 4, 499--503. 

\bibitem{Milnor}
J. Milnor; 
A unique decomposition theorem for $3$--manifolds, 
Amer.\ J.\ Math.\ \textbf{84} (1962), 1--7.

\bibitem{MB}
J.\ Morgan and H.\ Bass (eds.); 
The Smith conjecture, 
Pure and Applied Mathematics, vol. 112,  
Academic Press, 1984. 

\bibitem{MT2}
K. Motegi and M. Teragaito, 
Generalized torsion elements and bi-orderability of 3-manifold groups,  
Canad.\ Math.\ Bull.\ \textbf{60} (2017), 830--844.  

\bibitem{MR1}
R. Mura and A. Rhemtulla,
\textit{Solvable $R^*$--groups}, 
Math.\ Z.\  \textbf{142} (1975), 293--298. 

\bibitem{MR2}
R. Mura and  A. Rhemtulla,  
Orderable groups,
Lecture Notes in Pure and Applied Mathematics, Vol. \textbf{27}.
Marcel Dekker, Inc., New York-Basel, 1977.

\bibitem{Papa}
D. Papakyriakopoulos, 
On Dehn's lemma and the sphericity of knots, 
Ann.\ Math.\ \textbf{66} (1957), 1--26. 

\bibitem{Thurston}
W.P. Thurston; 
Three dimensional manifolds, Kleinian groups and hyperbolic geometry, 
Bull.\ Amer.\ Math.\ Soc.\ \textbf{6} (1982), 357--381. 

\end{thebibliography}
\end{document}